\documentclass[12pt]{article}
\makeatletter
\newcommand{\rmnum}[1]{\romannumeral #1}
\newcommand{\Rmnum}[1]{\expandafter\@slowromancap\romannumeral #1@}
\makeatother
\usepackage{}
\usepackage[numbers,sort&compress]{natbib}

\usepackage{color}
\usepackage{threeparttable}

\usepackage[english]{babel}
\usepackage{diagbox}
\usepackage{t1enc}
\usepackage{amsthm,amsmath,amssymb}
\usepackage{mathrsfs}
\usepackage[mathscr]{eucal}
\usepackage[latin1]{inputenc}
\usepackage[english]{babel}
\usepackage{latexsym}
\usepackage{graphicx}
\usepackage[noend]{algpseudocode}
\usepackage{algorithmicx,algorithm}
\usepackage{booktabs}
\usepackage{multirow}
\newtheorem{theorem}{Theorem}
\newtheorem{corollary}{Corollary}
\newtheorem{proposition}{Proposition}
\newtheorem{lemma}{Lemma}
\newtheorem{remark}{Remark}

\newtheorem{definition}{Definition}

\def \diag{\textup{diag}}

\def \T{\textup{T}}

\def\rank{\textup{rank}}

\begin{document}
\title{A characterization of generalized cospectrality of rooted graphs with applications in graph reconstruction}
\author{Wei Wang$^a$\thanks{Corresponding author. Email address: wangwei.math@gmail.com}
	\quad Wenqiang Wen$^a$
	 \quad Songlin Guo$^b$
\\
{\footnotesize $^{\rm a}$ School of Mathematics, Physics and Finance, Anhui Polytechnic University, Wuhu 241000, P. R. China}\\
{\footnotesize $^{\rm b}$ School of General Education, Wanjiang University of Technology, Ma'anshan, 243031, P. R. China}\\
}
\date{}

\maketitle
\begin{abstract}
Extending a classic result of Johnson and Newman, this paper provides a matrix characterization for two generalized cospectral graphs with a pair of generalized cospectral vertex-deleted subgraphs.  As an application, we present a new condition for the reconstructibility of a graph. In particular, we show	 that  a graph $G$ with at least three vertices is reconstructible if there exists a vertex-deleted subgraph that is almost controllable and has a nontrivial automorphism.	
\end{abstract}

	

\section{Introduction}
Let $G$ be an $n$-vertex graph  with adjacency matrix $A(G)$. The spectrum of $G$ refers to the multiset of  eigenvalues of $A(G)$. Two graphs $G$ and $H$ are \emph{cospectral} if they share the same spectrum. It is known that if $G$ and $H$ are cospectral then there exists an orthogonal matrix $Q$ such that $Q^\T A(G) Q=A(H)$. 

We are interested in two kinds of enhancements of the ordinary cospectrality: rooted-cospectrality and generalized cospectrality.  Let $(G,u)$ be a rooted graph with $u$ being the root vertex. We say two rooted graphs $(G,u)$ and $(H,v)$ are cospectral if (1) $G$ and $H$ are cospectral and (2) $G-u$ and $H-v$ are also cospectral. It turns out that the rooted-cospectrality of graphs can be characterized by specific orthogonal matrices as described in the following theorem. Without loss of generality, we may assume the root vertices are labeled as the last vertices in graphs.

\begin{theorem}[\cite{schwenk,liu2020}]
Let $G$ and $H$ be two $n$-vertex graphs with vertex set $\{u_1,\ldots,u_n\}$  and $\{v_1,\ldots,v_n\}$. Then the following are equivalent:

\noindent\textup{(\rmnum{1})}  $(G,u_n)$ and $(H,v_n)$ are cospectral.
	
\noindent\textup{(\rmnum{2})} There exists an orthogonal matrix of the form 
$
	\begin{pmatrix}
		Q&O\\
		O&1
	\end{pmatrix}
$
such that 
\begin{equation*}
	\begin{pmatrix}
	Q^\T&O\\
	O&1
\end{pmatrix}A(G)	\begin{pmatrix}
Q&O\\
O&1
\end{pmatrix}=A(H).
\end{equation*}
\end{theorem}
We say two graphs $G$ and $H$ are \emph{generalized cospectral} if $G$ and $H$ are cospectral with cospectral complement. Similar to  rooted-cospectrality, generalized cospectrality can be characterized by a special kind of orthogonal matrices. An orthogonal matrix is \emph{regular} if each row sum is 1. 
\begin{theorem}[\cite{johnson1980JCTB}]
	Let $G$ and $H$ be two graphs. Then the following are equivalent:
	
\noindent\textup{(\rmnum{1})} $G$ and $H$ are generalized cospectral.

\noindent\textup{(\rmnum{2})} There exists a regular orthogonal matrix $Q$ such that $Q^\T A(G) Q=A(H)$.
\end{theorem} 

The main aim of this paper is to unify the above two theorems. We say two rooted graphs $(G,u)$ and $(H,v)$ are generalized cospectral if (1) $G$ and $H$ are generalized cospectral and (2) $G-u$ and $H-v$ are also generalized cospectral. The main result of this paper is the following theorem.
\begin{theorem}\label{main}
	Let $G$ and $H$ be two $n$-vertex graphs with vertex set $\{u_1,\ldots,u_n\}$  and $\{v_1,\ldots,v_n\}$. Then the following are equivalent:
	
	\noindent\textup{(\rmnum{1})}  $(G,u_n)$ and $(H,v_n)$ are generalized cospectral, i.e., four graphs $G,\overline{G},G-u_n$ and $ \overline{G-u_n}$ are cospectral with $H,\overline{H},H-v_n$ and  $\overline{H-v_n}$, respectively. 
	
	\noindent\textup{(\rmnum{2})}There exists a regular orthogonal matrix of the form 
	$
	\begin{pmatrix}
		Q&O\\
		O&1
	\end{pmatrix}
	$
	such that 
	\begin{equation*}
		\begin{pmatrix}
			Q^\T&O\\
			O&1
		\end{pmatrix}A(G)	\begin{pmatrix}
			Q&O\\
			O&1
		\end{pmatrix}=A(H).
	\end{equation*}
\end{theorem}
The proof of Theorem \ref{main} will be given in Section 2. We remark that Theorem 3 was  reported by Farrugia \cite{farrugia} in  a different but essentially equivalent form, using the notion of \emph{overgraphs}. However, the original proof presented in \cite{farrugia} contains some inaccuracies, and  the present proof can be regarded as a correction. In Section 3, we apply Theorem \ref{main} in the problem of graph reconstruction. A new condition is obtained for the reconstructibility of  graphs, which improves the previous result of Hong \cite{hong}.

\section{Proof of Theorem \ref{main}}
Let $G$ be a graph with vertex set $\{u_1,\ldots,u_n\}$. Let $b$ be an $n$-dimensional $(0,1)$-vector. Following Farrugia \cite{farrugia}, the \emph{overgraph} $G+b$ is the graph whose adjacency matrix is $$A(G+b)=\begin{pmatrix}A(G)&b\\b^\T& 0\end{pmatrix}.$$

Let $A$ be an $n\times n$ real symmetric matrix and $\lambda_1,\ldots,\lambda_m$ be its distinct eigenvalues with multiplicities $k_1,\ldots, k_m$, respectively. Let $P_i$ be any $n\times k_i$ matrix whose columns consist of an orthonormal basis of $\mathcal{E}_{\lambda_i}(A)$, the eigenspace of $A$ corresponding to $\lambda_i$. Then $A$ has the spectral decomposition 
$$A=\lambda_1 P_1P_1^\T+\cdots+\lambda_m P_mP_m^\T.$$
We note that  $P_iP_i^\T$ is well defined although $P_i$ is not unique. Indeed, if $\tilde{P}_i$ consists of another orthogonal basis of $\mathcal{E}_{\lambda_i}(A)$, then we must have $\tilde{P}_i=P_iQ$ for some orthogonal matrix $Q$, which clearly implies $\tilde{P}_i\tilde{P}^\T_i= P_iP_i^\T$.

For a graph $G$, we use $\chi(G;x)$ to denote the characteristic polynomial of $G$, i.e., $\chi(G;x)=\chi(A(G);x)=\det(xI-A(G))$. We use $e^{(n)}$ (or $e$) to denote the all-ones  vector of dimension $n$.

\begin{lemma}[\cite{crs}]\label{basic}
Let $G$ be an $n$-vertex graph whose adjacency matrix $A$ has spectral decomposition $A=\lambda_1 P_1P_1^\T+\cdots+\lambda_m P_mP_m^\T$. Then 

\noindent\textup{(\rmnum{1})}	$\chi(G+b;x)=\chi(G;x)\left(x-\sum_{i=1}^m \frac{||P_i^\T b||^2}{x-\lambda_i}\right)$
for any  $b\in \{0,1\}^n$. 

\noindent\textup{(\rmnum{2})}  $\chi(\overline{G};x)=(-1)^n\chi(G;-x-1)\left(1-\sum_{i=1}^m \frac{||P_i^\T e||^2}{x+1+\lambda_i}\right).$
\end{lemma}
Suppose that $\chi(G;x)$ is known, then it is not difficult to see from Lemma \ref{basic} that knowledge of $\chi(G+e;x)$ is equivalent to knowledge of $\chi(\overline{G};x)$. This yields an equivalent definition on generalized cospectrality:
\begin{corollary}\label{gc2}
	Two graphs $G$ and $H$ are generalized cospectral if and only if two graphs $G$ and $G+e$ are cospectral with $H$ and $H+e$, respectively.
\end{corollary}
We need a technical lemma.
\begin{lemma}\label{tl}
	Let $\lambda_1,\ldots,\lambda_m$ be $m$ distinct real numbers. Let $a_1,a_2,\ldots, a_m,b_1,b_2,\ldots,b_m\in \mathbb{R}$. If 
	\begin{equation}\label{ab2}
		\left(1+\sum_{i=1}^{m}\frac{a_i}{\lambda-\lambda_i}\right)^2=\left(1+\sum_{i=1}^{m}\frac{b_i}{\lambda-\lambda_i}\right)^2 \text{~for~} \lambda\in \mathbb{R}\setminus\{\lambda_1,\ldots,\lambda_m\}
	\end{equation}
	then $a_i=b_i$ for $i\in \{1,\ldots,m\}$.
\end{lemma}
\begin{proof}
	By Eq.~\eqref{ab2}, we have 
	$$	\left(1+\sum_{i=1}^{m}\frac{a_i}{\lambda-\lambda_i}\right)=\pm \left(1+\sum_{i=1}^{m}\frac{b_i}{\lambda-\lambda_i}\right).$$
	We claim that the sign `$\pm$' can take `$-$' for at most $m$  different values of the variable $\lambda$. Indeed, suppose $\left(1+\sum_{i=1}^{m}\frac{a_i}{\lambda-\lambda_i}\right)=- \left(1+\sum_{i=1}^{m}\frac{b_i}{\lambda-\lambda_i}\right)$. By multiplying both sides by $\prod_{i=1}^m (\lambda-\lambda_i)$ and rearranging the terms, we obtain \begin{equation}\label{pm}
		2\prod_{i=1}^m (\lambda-\lambda_i)+\sum_{i=1}^m (a_i+b_i)\prod_{j=1,j\neq i}^{m}(\lambda-\lambda_j)=0.
	\end{equation}
	Noting that the left-hand side is a polynomial of degree $m$, we find that Eq.~\eqref{pm} has at most $m$ roots. This proves the claim. 
	
	By the claim, we see that the equality $1+\sum_{i=1}^{m}\frac{a_i}{\lambda-\lambda_i}= 1+\sum_{i=1}^{m}\frac{b_i}{\lambda-\lambda_i}$, or equivalently, \begin{equation}\label{ab}
	\sum_{i=1}^{m}\frac{a_i-b_i}{\lambda-\lambda_i}=0
	\end{equation}
	 holds for all $\lambda\in \mathbb{R}$ except finite number of values. Taking $\lambda\to \lambda_i$ in Eq.~\eqref{ab}, we easily find that $a_i=b_i$. This completes the proof.
\end{proof}
\begin{proposition}\label{rst}
	Let $G$ and $H$ be two $n$-vertex graph and $b,c$ be two vectors in $\{0,1\}^n$. If (1) $G+b$ and $H+c$ are  generalized cospectral  and (2) $G$ and $H$ are also generalized cospectral, then there exists an orthogonal matrix $Q$ such that $Q^\T A(G)Q=A(H)$, $Q^\T e=e$ and $Q^\T b=c$.
\end{proposition}
\begin{proof}
	Let $A=A(G)$ and $B=A(H)$. As $A$ and $B$ are cospectral, we may write the spectral decompositions of $A$ and $B$ as follows:
	\begin{equation}
A=\sum_{i=1}^m \lambda_i P_iP_i^\T \quad \text{and}\quad B=\sum_{i=1}^m \lambda_i R_iR_i^\T,
	\end{equation}
	where each $P_i$ and  $R_i$ consist of orthogonal bases  of $\mathcal{E}_{\lambda_i}(A)$ and $\mathcal{E}_{\lambda_i}(B)$, respectively.
		
	\noindent \textbf{Claim}:  $||P_i^\T b||=||R_i^\T c||$, $||P_i^\T e||=||R_i^\T e||$ and $\langle P_i^\T b, P_i^\T e\rangle=\langle R_i^\T c, R_i^\T e\rangle$ for $i=1,\ldots,m$.
	
	The first two equalities of the claim are easy consequences of Lemma \ref{basic}. Indeed, noting that $\chi(G+b;x)=\chi(H+c;x)$, it follows from  Lemma \ref{basic} (\rmnum{1}) that
	$$\chi(G;x)\left(x-\sum_{i=1}^m \frac{||P_i^\T b||^2}{x-\lambda_i}\right)=\chi(H;x)\left(x-\sum_{i=1}^m \frac{||R_i^\T c||^2}{x-\lambda_i}\right).$$
	As $\chi(G;x)=\chi(H;x)$, we must have 
		$$x-\sum_{i=1}^m \frac{||P_i^\T b||^2}{x-\lambda_i}=x-\sum_{i=1}^m \frac{||R_i^\T c||^2}{x-\lambda_i},$$
		which clearly implies $||P_i^\T b||=||R_i^\T c||$. A similar argument show that $||P_i^\T e||=||R_i^\T e||$. It remains to show the last equality $\langle P_i^\T b, P_i^\T e\rangle=\langle R_i^\T c, R_i^\T e\rangle$.
	
	Let $\hat{G}=(G+b)+e^{(n+1)}$ and $\hat{H}=(H+c)+e^{(n+1)}$. By Corollary \ref{gc2} and the first condition of this proposition, we see that  $\hat{G}$ and $\hat{H}$ are cospectral.  Note that the adjacency matrix of $\hat{G}$ is 
	\begin{equation*}
		\hat{A}=\begin{pmatrix}
			A&b&e\\
			b^\T& 0&1\\
			e^\T& 1&0
			\end{pmatrix}.
	\end{equation*}
Direct calculation shows that
\begin{align}\label{hG}
		\chi(\hat{G};x)&=\begin{vmatrix}
			\lambda I-A&-b&-e\\
			-b^\T &\lambda&-1\\
			-e^\T&-1&\lambda
		\end{vmatrix} \nonumber\\[5pt]
		&=\begin{vmatrix}\begin{pmatrix}
			\lambda I-A&-b&-e\\
			-b^\T &\lambda&-1\\
			-e^\T&-1&\lambda
		\end{pmatrix}\begin{pmatrix}
	I&(\lambda I-A)^{-1}b&(\lambda I-A)^{-1}e\\
		O &1&0\\
		O&0&1
		\end{pmatrix}\end{vmatrix}\nonumber\\[5pt]
		&=\begin{vmatrix}
			\lambda I-A&O&O\\
		-b^\T &\lambda-b^\T (\lambda I-A)^{-1}b&-1-b^\T (\lambda I-A)^{-1}e \nonumber\\
			-e^\T&-1-e^\T (\lambda I-A)^{-1}b&\lambda-e^\T (\lambda I-A)^{-1}e
		\end{vmatrix}\nonumber\\[5pt]
			&=|\lambda I-A|\begin{vmatrix}
			\lambda-\sum\limits_{i=1}^m\frac{1}{\lambda-\lambda_i}b^\T P_i P_i^\T b &-\left(1+\sum\limits_{i=1}^m \frac{1}{\lambda-\lambda_i}b^\T P_i P_i^\T e\right)\\
			-\left(1+\sum\limits_{i=1}^m \frac{1}{\lambda-\lambda_i}e^\T P_i P_i^\T b\right)&	\lambda-\sum\limits_{i=1}^m\frac{1}{\lambda-\lambda_i}e^\T P_i P_i^\T e\end{vmatrix} \nonumber\\[5pt]
			&=\chi(G;x)\begin{vmatrix}\lambda-\sum\limits_{i=1}^m\frac{||P_i^\T b||^2}{\lambda-\lambda_i}&-\left(1+\sum\limits_{i=1}^m\frac{\langle P_i^\T b,P_i^\T e\rangle}{\lambda-\lambda_i}\right)\\
			-\left(1+\sum\limits_{i=1}^m\frac{\langle P_i^\T b,P_i^\T e\rangle}{\lambda-\lambda_i}\right)&\lambda-\sum\limits_{i=1}^m\frac{||P_i^\T e||^2}{\lambda-\lambda_i}
		\end{vmatrix}.
	\end{align}
Similarly, we have 
\begin{equation}\label{hH}
	\chi(\hat{H};x)=\chi(H;x)\begin{vmatrix}\lambda-\sum\limits_{i=1}^m\frac{||R_i^\T c||^2}{\lambda-\lambda_i}&-\left(1+\sum\limits_{i=1}^m\frac{\langle R_i^\T c,R_i^\T e\rangle}{\lambda-\lambda_i}\right)\\
		-\left(1+\sum\limits_{i=1}^m\frac{\langle R_i^\T c,R_i^\T e\rangle}{\lambda-\lambda_i}\right)&\lambda-\sum\limits_{i=1}^m\frac{||R_i^\T e||^2}{\lambda-\lambda_i}
		\end{vmatrix}.
\end{equation}
As $\chi(\hat{G};x)=\chi(\hat{H};x)$, $\chi(G;x)=\chi(H;x)$, $||P_i^\T b||=||R_i^\T c||$ and $||P_i^\T e||=||R_i^\T e||$, combining Eqs. \eqref{hG} and \eqref{hH} leads to 
\begin{equation*}
	\left(1+\sum\limits_{i=1}^m\frac{\langle P_i^\T b,P_i^\T e\rangle}{\lambda-\lambda_i}\right)^2=\left(1+\sum\limits_{i=1}^m\frac{\langle R_i^\T c,R_i^\T e\rangle}{\lambda-\lambda_i}\right)^2,
\end{equation*}
It follows from Lemma \ref{tl} that $\langle P_i^\T b, P_i^\T e\rangle=\langle R_i^\T c, R_i^\T e\rangle$ for $i=1,\ldots,m$. This completes the proof of the Claim.

For each $i\in\{1,\dots,m\}$, let $k_i$ be the multiplicity of the eigenvalue $\lambda_i$ of $A$ (or $B$). Note that $P_i^\T b, P_i^\T e, R_i^\T c, R_i^\T e\in \mathbb{R}^{k_i}$. By the Claim, we see that for each $i$ there exists an orthogonal matrix $Q_i$ of order $k_i$ such that $Q_i(P_i^\T b)=R_i^\T c$ and $Q_i(P_i^\T e)=R_i^\T e$, see e.g. \cite[Theorem 7.3.11]{horn2013}.  Written  in the form of  block matrices, we have
\begin{equation}\label{block}
	\begin{pmatrix}
		Q_1 P_1^\T\\
		\vdots\\
		Q_mP_m^\T
	\end{pmatrix}(b,e)=\begin{pmatrix}
	R_1^\T\\
	\vdots\\
R_m^\T
	\end{pmatrix}(c,e).
\end{equation} 
Define 
\begin{equation}\label{defQ}
	Q=(P_1Q_1^\T,\ldots,P_mQ_m^\T)\begin{pmatrix}
		R_1^\T\\ \vdots\\R_m^\T
	\end{pmatrix}.
\end{equation}
It is easy to see that both matrices on the right-hand side of Eq.~\eqref{defQ} are orthogonal, implying $Q$ is an orthogonal matrix. By Eq.~\eqref{block}, we have $Q^\T (b,e)=(c,e)$. Finally, noting that 
\begin{equation*}
(P_1Q_1^\T,\ldots,P_mQ_m^\T)^\T A(P_1Q_1^\T,\ldots, P_m Q_m^\T)=\diag(\lambda_1I_{k_1},\ldots, \lambda_{m}I_{k_m})
\end{equation*}
and 
\begin{equation*}
	(R_1,\ldots,R_m)^\T B (R_1,\ldots, R_m)=\diag(\lambda_1I_{k_1},\ldots, \lambda_{m}I_{k_m})
\end{equation*}
we obtain
\begin{eqnarray*}
	Q^\T A Q&=&(R_1,\ldots,R_m)(P_1Q_1^\T,\ldots,P_mQ_m^\T)^\T A(P_1Q_1^\T,\ldots, P_m Q_m^\T)\begin{pmatrix}
		R_1^\T\\ \vdots\\R_m^\T
	\end{pmatrix}\\
	&=&(R_1,\ldots,R_m)\diag(\lambda_1I_{k_1},\ldots, \lambda_{m}I_{k_m})\begin{pmatrix}
		R_1^\T\\ \vdots\\R_m^\T
	\end{pmatrix}\\
	&=&B.
\end{eqnarray*}
This completes the proof of this proposition.	
\end{proof}
\begin{remark}\normalfont{
	Proposition \ref{rst} was originally reported as Theorem 12 in \cite{farrugia}. In that paper, Farrugia essentially proved that there exist two orthogonal matrices $Q_1$ and $Q_2$ such that $Q_1^\T A(G)Q_1=Q_2^\T A(G)Q_2=A(H)$, $Q_1^\T e=e$ and $Q_2^\T b=c$. But  his claim that the two matrices $Q_1$ and $Q_2$ can be chosen equal requires more justification. The key of the current proof is the newly established equality $\langle P_i^\T b,P_i^\T e\rangle=\langle R_i^\T c,R_i^\T e\rangle$, which guarantees the realizability of $Q_1=Q_2$.}
\end{remark}
\noindent{\textbf{Proof of Theorem \ref{main}} The implication (\rmnum{2})$\implies$(\rmnum{1}) is straightforward; we only prove the other direction. Let $G_1=G-u_n$ and $H_1=H-v_n$. Then the adjacency matrices of $G$ and $H$ have the form
	\begin{equation*}
		A(G)=\begin{pmatrix}A(G_1)&b\\b^\T&0\end{pmatrix}
		\quad \text{and}\quad 	A(H)=\begin{pmatrix}A(H_1)&c\\c^\T&0\end{pmatrix}.
		\end{equation*}
Now the condition (\rmnum{1}) can be restated as that both the pair $G_1+b$, $H_1+c$ and the pair $G_1$, $H_1$ are generalized cospectral graphs. Noting that $G_1$ and $H_1$ contain $(n-1)$ vertices, it follows from Proposition \ref{rst} that there exists an orthogonal matrix $Q$ such that
$Q^\T A(G_1)Q=A(H_1)$ $Q^\T e^{(n-1)}=e^{(n-1)}$ and $Q^\T b =c$. Direct calculation shows that
\begin{equation*}
	\begin{pmatrix}
		Q^\T&O\\
		O&1
	\end{pmatrix}\begin{pmatrix}A(G_1)&b\\b^\T& 0\end{pmatrix}
		\begin{pmatrix}
		Q&O\\
		O&1
	\end{pmatrix}=\begin{pmatrix}Q_1^\T A(G_1)Q&Q^\T b\\b^\T Q&0\end{pmatrix}=\begin{pmatrix} A(H_1)&c\\c^\T &0\end{pmatrix}.
	\end{equation*}
As $Q$ is a regular orthogonal matrix, the block matrix 	$\begin{pmatrix}
	Q^\T&O\\
	O&1
\end{pmatrix}$ is also a regular orthogonal matrix. This completes the proof of Theorem \ref{main}.
\section{A new condition for graph reconstructibility}
Given a graph $G$ with vertex set $\{u_1,\ldots,u_n\}$, the \emph{deck} of $G$, denoted by $\mathcal{D}(G)$, is the multiset of its vertex-deleted (unlabeled) subgraphs $G-u_i$ for $i=1,\ldots,n$.  A graph $H$ is called a \emph{reconstruction} of $G$ if $\mathcal{D}(H)=\mathcal{D}(G)$. If every reconstruction of $G$ is isomorphic to
$G$, then $G$ is said to be \emph{reconstructible}. The Reconstruction Conjecture (also called Ulam Conjecture or Kelly-Ulam Conjecture) claims that every graph with at least three vertices  is reconstructible. 

An eigenvalue $\lambda$ of a graph $G$ is called a \emph{main eigenvalue} if the corresponding eigenspace is not orthogonal to the all-ones vector $e$. An $n$-vertex graph $G$ is \emph{controllable} (resp. \emph{almost controllable}) if $G$ has $n$ (resp. $n-1$) main eigenvalues. In a short note, Hong \cite{hong} obtained the following condition for reconstructibility,  based on the work of Godsil-McKay \cite{godsil}.
\begin{theorem}[\cite{hong}]\label{cons}
If there exists a vertex-deleted subgraph $G-u_i$ that is controllable, then $G$ is reconstructible.
\end{theorem}
The  aim of this section is to extend Theorem \ref{cons} to  almost controllable subgraphs under some mild restrictions. The main tool is Theorem \ref{main}, together with some basic results on almost controllable graphs.

For an $n$-vertex graph $G$ with adjacency matrix $A$, the \emph{walk matrix} of $G$ is 
\begin{equation*}
W(G):=[e,Ae,\ldots,A^{n-1}e].
\end{equation*}
It is known that the number of main eigenvalues of $G$ equals the rank of $W(G)$. We use $\mathcal{H}_n$ to denote the set of all almost controllable graphs on $n$ vertices. Clearly, if  $G\in \mathcal{H}_n$ then $\rank\, W(G)=n-1$. For a graph $G\in \mathcal{H}_n$, the following Householder matrix associated with $G$ is crucial.
\begin{definition}[\cite{qiu}]\label{defQ0}\normalfont{
	For $G\in \mathcal{H}_n$, define
	$$Q_0=Q_0(G)=I_n-\frac{2\xi\xi^\T}{\xi^\T\xi},$$
	where $\xi$ is a unique (up to the sign) nonzero integral vector $\xi=(a_1,a_2,\ldots,a_n)$ satisfying $W^\T(G) \xi=0$ and $\gcd(a_1, a_2,\ldots,a_n) =1$.}
\end{definition}
Clearly, $Q_0$ is a symmetric orthogonal matrix. Moreover, from the equation $W^\T(G)\xi=0$, we easily see that  $\xi^\T e=0$ and hence $Qe=e$, i.e., the orthogonal matrix $Q_0(G)$ is regular. Let $\textup{RO}(n)$ denote the group of regular orthogonal matrices of order $n$. Note that $\textup{RO}(n)$ contains the  group of $n\times n$ permutation matrices as a subgroup. The importance of $Q_0(G)$ can be described in the following theorem.
\begin{theorem}[\cite{wang,qiu}]\label{twosol}
For $G\in \mathcal{H}_n$,	the solution set of the matrix equation $Q^\T A(G)Q=A(G)$ with variable $Q\in \textup{RO}(n)$ is exactly $\{I_n,Q_0(G)\}$.
\end{theorem}
By Theorem \ref{twosol}, we see that graphs in $\mathcal{H}_n$ can be naturally partitioned into two subsets $\mathcal{H}_n^s$ and  $\mathcal{H}_n^a$, where $\mathcal{H}_n^s$ collects graphs $G$ in $\mathcal{H}_n$ for which $Q_0$ is a permutation matrix (a nontrivial automorphism of $G$), and $\mathcal{H}_n^a$ collects the asymmetric ones.  Now we can state and prove the main result of this section.
\begin{theorem}\label{ac}
	Let $n\ge 3$ and $G$ be an $n$-vertex graph with vertex set $\{u_1,\ldots,u_n\}$. Then $G$ is reconstructible if there exists a vertex-deleted subgraph, say $G-u_n$, satisfying either of the following two conditions:
	
\noindent \textup{(\rmnum{1})} $G-u_n\in \mathcal{H}_{n-1}^s$,

 \noindent \textup{(\rmnum{2})} $G-u_n\in   \mathcal{H}_{n-1}^a$ but $Q_0b\not\in \{0,1\}^{n-1}\setminus \{b\}$, where $Q_0=Q_0(G-u_n)$ and $b\in \{0,1\}^{n-1}$ is the indicator vector of the neighborhood of  $u_n$ in $G$.
\end{theorem}
\begin{proof}
Let $H$ be any reconstruction of $G$. It is known that $G$ and $H$ are generalized cospectral, see \cite{tutte} or \cite{hagos}. Let $A=A(G-u_n)$. Then the adjacency matrix of $G$ can be written as 
$$A(G)=\begin{pmatrix}
	A&b\\
	b^\T&0
\end{pmatrix}.
$$
By relabeling vertices in $H$ appropriately, we may assume the adjacency matrix of $H$ has the form
$$A(H)=\begin{pmatrix}
	A&c\\
	c^\T &0
\end{pmatrix}.
$$
It follows from Theorem \ref{main} that there exists a regular orthogonal matrix of the form $\begin{pmatrix} Q&O\\O&1\end{pmatrix}$ such that
\begin{equation}\label{qabc}
	\begin{pmatrix}
		Q^\T &O\\
		O&1
	\end{pmatrix}\begin{pmatrix}A&b\\b^\T&0\end{pmatrix}
	\begin{pmatrix}
	Q &O\\
	O&1
	\end{pmatrix}=\begin{pmatrix}A&c\\c^\T&0\end{pmatrix}.
\end{equation}
This means that  $Q\in \textup{RO}(n-1)$, $Q^\T A Q=A$, and $Q^\T b=c$. Let $Q_0=Q_0(G-u_n)$ be the Household matrix as given in Definition \ref{defQ0}.   If $G-u_n\in \mathcal{H}_{n-1}^s$, then by Theorem \ref{twosol} we have $Q\in \{I_{n-1},Q_0\}$, where $Q_0$ is a permutation matrix.  But this means $G$ and $H$ are isomorphic by Eq.~\eqref{qabc}. Now consider the other case  $G-u_n\in \mathcal{H}_{n-1}^a$. We may assume $c\neq b$ since otherwise $H=G$ and we are done. Thus, $c\in \{0,1\}^{n-1}\setminus \{b\}$ and hence $Q_0b\neq c$ by the condition of this theorem.  Consequently, $Q\neq Q_0$ and hence we must have $Q=I_{n-1}$ by Theorem \ref{twosol}.  This again forces $H=G$ by Eq.~\eqref{qabc}. Either case implies that $H$ is isomorphic to $G$. Thus $G$ is reconstructible, completing the proof of Theorem \ref{ac}.
\end{proof} 
 \section*{Acknowledgments}
This work is partially supported by the National Natural Science Foundation of China (Grant No. 12001006).

\end{document}